\documentclass[a4paper, 12pt, twoside]{article}
\usepackage[left=2.9cm,right=2.9cm, top=3.4cm,bottom=3.6cm,bindingoffset=0cm]{geometry}

\newenvironment{resetCounters}
{
\setcounter{section}{0}
\setcounter{subsection}{0}
\setcounter{subsubsection}{0}
\setcounter{paragraph}{0}
\setcounter{subparagraph}{0}
\setcounter{equation}{0}
\setcounter{figure}{0}
\setcounter{table}{0}
\setcounter{footnote}{0}
\setcounter{theorem}{0}
\setcounter{lemma}{0}
\setcounter{remark}{0}
\setcounter{corollary}{0}
\setcounter{definition}{0}
\setcounter{example}{0}
}{}

\usepackage[nooneline]{caption}
\usepackage{amsmath}
\usepackage[T2A, T1]{fontenc}
\usepackage[utf8]{inputenc}
\usepackage[russian,english]{babel}
\usepackage{amsthm}
\usepackage{graphicx}
\usepackage{textcomp}
\usepackage{amsfonts}
\usepackage{amssymb}
\usepackage{epsfig}
\usepackage{color}
\usepackage{multirow}
\usepackage{dsfont}

\newtheorem{theorem}{Theorem}

\newtheorem{definition}{Definition}

\begin{document}

\makeatletter
\renewcommand{\@oddhead}{\hbox to 151.5mm{\hrulefill\raisebox{2.2mm}{\underline{\strut{\small\slshape }}}}}
\renewcommand{\@evenhead}{\hbox to 151.5mm{\raisebox{2.2mm}{\underline{\strut{\small\slshape Volume X, Issue Y (201Z)}}}\hrulefill}}
\makeatother

\setcounter{page}{0}

\begin{center}
{\Large\bf Behavior and dynamics of the set of absolute nilpotent and idempotent
elements of chain of evolution algebras depending on the time}
\end{center}

\begin{center}
{\sc Anvar Imomkulov}\\
{\it  V.I.Romanovskiy Institute of Mathematics, Tashkent, Uzbekistan\\
}
e-mail: {\tt a.imomkulov@mathnet.uz}
\end{center}

\begin{abstract}
In this paper we construct some families of three-dimensional evolution algebras
which satisfies Chapman-Kolmogorov equation. For all of
these chains we study the behavior of the baric property, the behavior
of the set of absolute nilpotent elements and dynamics of the set of
idempotent elements depending on the time.

\emph{\textbf{Keywords:}Evolution algebras, Chapman-Kolmogorv equation, baric algebra,
property transition, idempotent, nilpotent.}

\emph{\textbf{Mathematics Subject Classification (2010):}13J30, 13M05}
\end{abstract}

\section*{Introduction}

Mathematical methods have long been successfully used in population genetics. Research from
mathematics to population genetics is based on Mendel's laws (Gregor Johann Mendel, 1822-1884),
where he used symbols that, from an algebraic point of view, suggest the expression of his genetic laws.
Between 1856 and 1863, Mendel studied pea hybridization and gave a fundamental concept of classical genetics,
a gene called "constant character", to explain the observed inheritance statistics. In 1866, Gregor Mendel
published the results of many years of experiments on the selection of pea plants \cite{Mendel}. He showed
that both parents must transmit discrete physical factors that convey information about their characteristics
to their offspring at conception. Mendel was the first to use symbols that, from an algebraic point of view,
express their genetic laws. Thus, mathematicians and geneticists once used non-associative algebras to study
Mendelian genetics, and later some other authors called it "Mendelian algebras."

A notion of evolution algebra is introduced by J.P.Tian in \cite{Tian}. This evolution algebra
is defined as follows. Let $(E,\cdot)$ be an algebra over a field $\mathbb K$. If it admits a
countable basis $e_{1},e_{2},\ldots,$ such that $e_{i}\cdot e_{j}=0,$ if $i\neq j$ and
$e_{i}\cdot e_{i}=\sum_{k}a_{ik}e_{k},$ for any $i$, then this algebra is called an evolution algebra.
The concept of evolution algebras lies between algebras and dynamical
systems. Algebraically, evolution algebras are non-associative Banach algebra;
dynamically, they represent discrete dynamical systems. Evolution
algebras have many connections with other mathematical fields including
graph theory, group theory, stochastic processes, mathematical physics, etc.

The Kolmogorov-Chapman equation gives the fundamental relationship
between the probability transitions (kernels). Namely, it is known that
(see e.g. \cite{SK}) if each element of a family of matrices satisfying the
Kolmogorov-Chapman equation is stochastic, then it generates a Markov process.

There are many random processes which can not be described by Markov processes
of square stochastic matrices (see for example \cite{Dmitriev,HT,Kampen,Maksimov}).

To have non-Markov process one can consider a solution of the Kolmogorov-Chapman
equation which is not stochastic for some time as in \cite{CLR,ORT,RMchicken}, where a chain
of evolution algebras (CEA) is introduced and investigated. Later, this notion
of CEA was generalized in \cite{LRflow}, where a concept of flow of arbitrary finite-dimensional
algebras (i.e. their matrices of structural constants are cubicmatrices) is introduced.
In \cite{CLR19} considered Markov processes of cubic stochastic (in a fixed sense) matrices
which are also called quadratic stochastic pro-cess (QSPs)

In the book \cite{Tian}, the foundation of evolution algebra theory and applications
in non-Mendelian genetics and Markov chains are developed.
In \cite{RT} the algebraic structures of function spaces defined by graphs and
state spaces equipped with Gibbs measures by associating evolution algebras
are studied. Results of \cite{RT} also allow a natural introduction of thermodynamics
in studying of several systems of biology, physics and mathematics
by theory of evolution algebras.
There exist several classes of non-associative algebras (baric, evolution,
Bernstein, train, stochastic, etc.), whose investigation has provided a number
of significant contributions to theoretical population genetics. In \cite{Imomkulov}
classified a family of three-dimensional evolution algebras, and in \cite{IR} and \cite{Iuzbmath} considered
a notion of approximation of an algebra by evolution algebras.
In \cite{CLR} a notion of a chain of evolution algebras is introduced.
This chain is a dynamical system the state of which at each given time is
an evolution algebra. The chain is defined by the sequence of matrices of
the structural constants (of the evolution algebras considered in \cite{Tian}) which
satisfies the Chapman-Kolmogorov equation. In \cite{RM} studied
chains generated by two-dimensional evolution algebras.

In this paper we continue investigation of chain of evolution algebras,
in more detail we study chains generated by three-dimensional evolution algebras.

\section{Preliminaries}
Following \cite{CLR} we consider a family $\{E^{[s,t]}: s,t\in\mathbb{R}, 0\leq s\leq t\}$ of
$n$-dimensional evolution algebras over the field $\mathbb R$, with basis $e_{1},e_{2},\dots,e_{n}$,
and the multiplication table $$e_{i}e_{i}=\sum_{j=1}^{n}a_{i,j}^{[s,t]}e_{j},
\,\,\,\,\,\,i=1,\dots,n; \,\,\,\,\,e_{i}e_{j}=0, \,\,\,i\neq j.$$
Here parameters $s,t$ are considered as time.

Denote by $\mathcal M_{[s,t]}=\left(a_{i,j}^{[s,t]}\right)_{i,j=1,\dots n}$ the matrix of structural constants.

\begin{definition}
A family $\{E^{[s,t]}: s,t\in\mathbb{R}, 0\leq s\leq t\}$ of
$n$-dimensional evolution algebras over the field $\mathbb R$,
is called a {\it chain of evolution algebras} CEA if the matrix $\mathcal M_{[s,t]}$
of structural constants satisfies the Chapman-Kolmogorov equation
\begin{equation}\label{Chapman-Kolmogorov}
\mathcal M_{[s,t]}=\mathcal M_{[s,\tau]}\mathcal M_{[\tau,t]}, \,\,\,for \,\,\,any \,\,\,s<\tau<t.
\end{equation}
\end{definition}

A character for an algebra $A$ is a nonzero multiplicative linear form on $A$,
that is, a nonzero algebra homomorphism from $A$ to $\mathbb{R}$ \cite{Lyubich}. Not every algebra
admits a character. For example, an algebra with the zero multiplication
has no character.

\begin{definition} A pair $(A,\sigma)$ consisting of an algebra $A$ and a character
$\sigma$ on $A$ is called a baric algebra. The homomorphism $\sigma$ is called the weight
(or baric) function of $A$ and $\sigma(x)$ the weight (baric value) of $x$.
\end{definition}

In \cite{Lyubich} for the evolution algebra of a free population it is proven that there is
a character $\sigma(x)=\sum_{i}x_{i}$, therefore that algebra is baric. But the evolution
algebra $E$ introduced in \cite{Tian} is not baric, in general. The following theorem
gives a criterion for an evolution algebra $E$ to be baric.

\begin{theorem}\label{barikliksharti}
An $n$-dimensional evolution algebra $E$, over the field $\mathbb{R}$, is baric if and only if
there is a column $(a_{1i_{0}},\ldots,a_{ni_{0}})^T$ of its structural
constants matrix $\mathcal{M}=(a_{ij})_{i,j=1,\ldots,n}$, such that $a_{i_{0}i_{0}}\neq0$ and
$a_{ii_{0}}=0,$ for all $i\neq i_{0}.$ Moreover, the corresponding weight function is $\sigma(x)=a_{i_{0}i_{0}}x_{i_{0}}.$
\end{theorem}

In \cite{CLR} several concrete examples of chains of evolution algebras are
given and time-dynamics are studied. We continue the research of CEAs, in more detail,
we study the CEAs generated by three-dimensional evolution algebras.

To construct a chain of three-dimensional evolution algebras one has to solve
equation (\ref{Chapman-Kolmogorov}) for the $3\times3$ matrix $\mathcal M_{[s,t]}$.
This equation gives the following system of functional equations (with nine unknown functions):
\begin{equation}\label{equation}\begin{array}{c}
a_{11}^{[s,t]}=a_{11}^{[s,\tau]}a_{11}^{[\tau,t]}+a_{12}^{[s,\tau]}a_{21}^{[\tau,t]}+a_{13}^{[s,\tau]}a_{31}^{[\tau,t]},\\[2mm]
a_{12}^{[s,t]}=a_{11}^{[s,\tau]}a_{12}^{[\tau,t]}+a_{12}^{[s,\tau]}a_{22}^{[\tau,t]}+a_{13}^{[s,\tau]}a_{32}^{[\tau,t]},\\[2mm]
a_{13}^{[s,t]}=a_{11}^{[s,\tau]}a_{13}^{[\tau,t]}+a_{12}^{[s,\tau]}a_{23}^{[\tau,t]}+a_{13}^{[s,\tau]}a_{33}^{[\tau,t]},\\[2mm]
a_{21}^{[s,t]}=a_{21}^{[s,\tau]}a_{11}^{[\tau,t]}+a_{22}^{[s,\tau]}a_{21}^{[\tau,t]}+a_{23}^{[s,\tau]}a_{31}^{[\tau,t]},\\[2mm]
a_{22}^{[s,t]}=a_{21}^{[s,\tau]}a_{12}^{[\tau,t]}+a_{22}^{[s,\tau]}a_{22}^{[\tau,t]}+a_{23}^{[s,\tau]}a_{32}^{[\tau,t]},\\[2mm]
a_{23}^{[s,t]}=a_{21}^{[s,\tau]}a_{13}^{[\tau,t]}+a_{22}^{[s,\tau]}a_{23}^{[\tau,t]}+a_{23}^{[s,\tau]}a_{33}^{[\tau,t]},\\[2mm]
a_{31}^{[s,t]}=a_{31}^{[s,\tau]}a_{11}^{[\tau,t]}+a_{32}^{[s,\tau]}a_{21}^{[\tau,t]}+a_{33}^{[s,\tau]}a_{31}^{[\tau,t]},\\[2mm]
a_{32}^{[s,t]}=a_{31}^{[s,\tau]}a_{12}^{[\tau,t]}+a_{32}^{[s,\tau]}a_{22}^{[\tau,t]}+a_{33}^{[s,\tau]}a_{32}^{[\tau,t]},\\[2mm]
a_{33}^{[s,t]}=a_{31}^{[s,\tau]}a_{13}^{[\tau,t]}+a_{32}^{[s,\tau]}a_{23}^{[\tau,t]}+a_{33}^{[s,\tau]}a_{33}^{[\tau,t]}.
\end{array}
\end{equation}
But analysis of the system (\ref{equation}) is difficult. In \cite{ORT} it was solved when
matrix $\mathcal M_{[s,t]}$ has upper-triangular view. We will consider several cases where
the system is solvable. For this solutions corresponds some evolution algebras.
And we will check which of these algebras are isomorphic to 3-dimensional
classified algebras.

\section{Construction of chains of three-dimensional evolution algebras}
In this section we find solves of system of equations (\ref{equation}) in particular cases.

\subsection{Three-dimensional CEAs corresponding to matrices with same columns}
Let $a_{11}^{[s,t]}=a_{12}^{[s,t]}=a_{13}^{[s,t]}=\alpha(s,t)$,
$a_{21}^{[s,t]}=a_{22}^{[s,t]}=a_{23}^{[s,t]}=\beta(s,t)$,
$a_{31}^{[s,t]}=a_{32}^{[s,t]}=a_{33}^{[s,t]}=\gamma(s,t)$. Then equation (\ref{equation}) reduced to
\begin{equation}\label{equation1}\begin{array}{c}
\alpha(s,t)=\alpha(s,\tau)\left(\alpha(\tau,t)+\beta(\tau,t)+\gamma(\tau,t)\right),\\
\beta(s,t)=\beta(s,\tau)\left(\alpha(\tau,t)+\beta(\tau,t)+\gamma(\tau,t)\right),\\
\gamma(s,t)=\gamma(s,\tau)\left(\alpha(\tau,t)+\beta(\tau,t)+\gamma(\tau,t)\right).
\end{array}
\end{equation}

Denote
\begin{equation}\label{denote1}\begin{array}{c}
\delta(s,t)=\alpha(s,t)+\beta(s,t)+\gamma(s,t),\\
\zeta(s,t)=\alpha(s,t)-\beta(s,t)+\gamma(s,t),\\
\eta(s,t)=\alpha(s,t)-\beta(s,t)-\gamma(s,t).
\end{array}
\end{equation}
Then the last system of functional equations can be written as
$$\delta(s,t)=\delta(s,\tau)\delta(\tau,t),\,\,\,\,\,
\zeta(s,t)=\zeta(s,\tau)\delta(\tau,t),\,\,\,\,\,
\eta(s,t)=\eta(s,\tau)\delta(\tau,t),\,\,\,\,\, s\leq\tau\leq t. $$

The first equation is Cantor's second equation which has a very rich family of solutions:
\begin{itemize}
\item[a)] $\delta(s,t)\equiv0$;
\item[b)] $\delta(s,t)=\frac{h(t)}{h(s)}$, where $h$ is an arbitrary function with $h(s)\neq0$;
\item[c)] $\delta(s,t)=\left\{\begin{array}{l}
1,\,\,\,\,\, \text{if}\,\,\,\,\, s\leq t<a,\\
0,\,\,\,\,\, \text{if}\,\,\,\,\, t\geq a,\end{array}\right.$ where $a>0$.
\end{itemize}

Using these solutions from the second and third equations we find $\zeta$ and $\eta$:
\begin{itemize}
\item[a')] $\zeta(s,t)=\eta(s,t)=0$;
\item[b')] $\zeta(s,t)=g(s)h(t)$, $\zeta(s,t)=f(s)h(t)$ where $g$ and $f$
are an arbitrary functions;
\item[c')] $\zeta(s,t)=\left\{\begin{array}{l}
\varphi(s),\,\,\,\,\, \text{if}\,\,\,\,\, s\leq t<a,\\
0,\,\,\,\,\,\,\,\,\,\,\,\,\, \text{if}\,\,\,\,\, t\geq a,\end{array}\right.$
$\eta(s,t)=\left\{\begin{array}{l}
\psi(s),\,\,\,\,\, \text{if}\,\,\,\,\, s\leq t<a,\\
0,\,\,\,\,\,\,\,\,\,\,\,\,\, \text{if}\,\,\,\,\, t\geq a,\end{array}\right.$ where
$a>0$ and $\varphi(s)$, $\psi(s)$ are an arbitrary functions.
\end{itemize}

Substituting these solutions into (\ref{denote1}) and finding $\alpha(s,t)$, $\beta(s,t)$
and $\gamma(s,t)$ we get the following matrices
$$\mathcal M_{0}^{[s,t]}=\left(\begin{array}{cccccc}
0 & 0 & 0\\
0 & 0 & 0\\
0 & 0 & 0\end{array}\right);$$
$$\mathcal M_{1}^{[s,t]}=\frac{1}{2}\left(\begin{array}{cccccc}
h(t)\left(\frac{1}{h(s)}+f(s)\right) & h(t)\left(\frac{1}{h(s)}+f(s)\right) & h(t)\left(\frac{1}{h(s)}+f(s)\right)\\[2mm]
h(t)\left(\frac{1}{h(s)}-g(s)\right) & h(t)\left(\frac{1}{h(s)}-g(s)\right) & h(t)\left(\frac{1}{h(s)}-g(s)\right)\\[2mm]
h(t)\left(g(s)-f(s)\right) & h(t)\left(g(s)-f(s)\right) & h(t)\left(g(s)-f(s)\right)\end{array}\right);$$
$$\mathcal M_{2}^{[s,t]}=\frac{1}{2}\left\{\begin{array}{ll}\left(\begin{array}{cccccc}
1+\psi(s) & 1+\psi(s) & 1+\psi(s)\\[2mm]
1-\varphi(s) & 1-\varphi(s) & 1-\varphi(s)\\[2mm]
\varphi(s)-\psi(s) & \varphi(s)-\psi(s) & \varphi(s)-\psi(s)\end{array}\right),\,\,\,\,\, \text{if}\,\,\,\,\, s\leq t<a,\\[2mm]
\left(\begin{array}{cccccc}
0 & 0 & 0\\
0 & 0 & 0\\
0 & 0 & 0\end{array}\right),\,\,\,\,\,\,\,\,\,\,\,\,\,\,\,\,\,\,\,\,\,\,\,\,\,\,\,\,
\,\,\,\,\,\,\,\,\,\,\,\,\,\,\,\,\,\,\,\,\,\,\,\,\,\,\,\,\,\,\,\,\,\,\,\,\,\,\,\,\,\,\,
\,\,\,\,\,\,\,\,\,\,\,\,\,\,\,\, \text{if}\,\,\,\,\,t\geq a.
\end{array}\right.;$$

 Thus in this case we have three CEAs: $E_{i}^{[s,t]},\,\,0\leq s\leq t$, which correspond to the
 $\mathcal M_{i}^{[s,t]},\,\, i=0,1,2$ listed above.

\subsection{Three-dimensional CEAs with proportional rows.}
In the joint paper (accepted by journal Filomat on November) with M.V.Velasco we have described this case.
And we obtained chain as following:
$$\mathcal M_{3}^{[s,t]}=\frac{1}{\Phi(s)}\left(\begin{array}{ccc}
g_{1}(t) & g_{2}(t) & g_{3}(t)\\[2mm]
\psi(s)g_{1}(t) & \psi(s)g_{2}(t) & \psi(s)g_{3}(t)\\[2mm]
\varphi(s)g_{1}(t) & \varphi(s)g_{2}(t) & \varphi(s)g_{3}(t)\end{array}\right).$$
where, $\Phi(s)=g_{1}(s)+\psi(s)g_{2}(s)+\varphi(s)g_{3}(s)\neq0,$ and $g_1,g_2,g_3,\psi,\varphi$ are arbitrary functions.

Thus in this case we have one CEA: $E_{3}^{[s,t]}$, which correspond to the
 $\mathcal M_{3}^{[s,t]}$ listed above.

\subsection{Three-dimensional CEAs in some other cases}

\textbf{Case 1.} $a_{11}^{[s,t]}=a_{12}^{[s,t]}=a_{21}^{[s,t]}=0.$ In this case system of equations
(\ref{equation}) has the following view:
\begin{equation}\label{equation3}\begin{array}{l}
0=a_{13}^{[s,\tau]}a_{31}^{[\tau,t]},\\[2mm]
0=a_{13}^{[s,\tau]}a_{32}^{[\tau,t]},\\[2mm]
a_{13}^{[s,t]}=a_{13}^{[s,\tau]}a_{33}^{[\tau,t]},\\[2mm]
0=a_{23}^{[s,\tau]}a_{31}^{[\tau,t]},\\[2mm]
a_{22}^{[s,t]}=a_{22}^{[s,\tau]}a_{22}^{[\tau,t]}+a_{23}^{[s,\tau]}a_{32}^{[\tau,t]},\\[2mm]
a_{23}^{[s,t]}=a_{22}^{[s,\tau]}a_{23}^{[\tau,t]}+a_{23}^{[s,\tau]}a_{33}^{[\tau,t]},\\[2mm]
a_{31}^{[s,t]}=a_{33}^{[s,\tau]}a_{31}^{[\tau,t]},\\[2mm]
a_{32}^{[s,t]}=a_{32}^{[s,\tau]}a_{22}^{[\tau,t]}+a_{33}^{[s,\tau]}a_{32}^{[\tau,t]},\\[2mm]
a_{33}^{[s,t]}=a_{31}^{[s,\tau]}a_{13}^{[\tau,t]}+a_{32}^{[s,\tau]}a_{23}^{[\tau,t]}+a_{33}^{[s,\tau]}a_{33}^{[\tau,t]}.
\end{array}
\end{equation}

\textbf{Case 1.1.} $a_{13}^{[s,\tau]}=a_{23}^{[s,\tau]}=0$. Consequently $a_{13}^{[s,t]}=0$.
Then system of equation (\ref{equation3}) has the following view:
\begin{equation}\label{equation4}\begin{array}{l}
a_{22}^{[s,t]}=a_{22}^{[s,\tau]}a_{22}^{[\tau,t]},\\[2mm]
a_{23}^{[s,t]}=a_{22}^{[s,\tau]}a_{23}^{[\tau,t]},\\[2mm]
a_{31}^{[s,t]}=a_{33}^{[s,\tau]}a_{31}^{[\tau,t]},\\[2mm]
a_{32}^{[s,t]}=a_{32}^{[s,\tau]}a_{22}^{[\tau,t]}+a_{33}^{[s,\tau]}a_{32}^{[\tau,t]},\\[2mm]
a_{33}^{[s,t]}=a_{32}^{[s,\tau]}a_{23}^{[\tau,t]}+a_{33}^{[s,\tau]}a_{33}^{[\tau,t]}.
\end{array}
\end{equation}

The first equation of (\ref{equation4}) has the following solutions:
\begin{itemize}
\item[a)] $a_{22}^{[s,t]}\equiv0$;
\item[b)] $a_{22}^{[s,t]}=\frac{h(t)}{h(s)}$, where $h$ is an arbitrary function with $h(s)\neq0$;
\item[c)] $a_{22}^{[s,t]}=\left\{\begin{array}{l}
1,\,\,\,\,\, \text{if}\,\,\,\,\, s\leq t<a,\\
0,\,\,\,\,\, \text{if}\,\,\,\,\, t\geq a,\end{array}\right.$ where $a>0$.
\end{itemize}

\textbf{Case 1.1.a.} $a_{22}^{[s,t]}\equiv0.$ Consequently,  $a_{23}^{[s,t]}=0.$
Then we the following system of equations:
\begin{equation}\label{equation5}\begin{array}{c}
a_{31}^{[s,t]}=a_{33}^{[s,\tau]}a_{31}^{[\tau,t]},\\[2mm]
a_{32}^{[s,t]}=a_{33}^{[s,\tau]}a_{32}^{[\tau,t]},\\[2mm]
a_{33}^{[s,t]}=a_{33}^{[s,\tau]}a_{33}^{[\tau,t]}.
\end{array}
\end{equation}

The last equation of (\ref{equation5}) has the following solutions:

\begin{itemize}
\item[a')] $a_{33}^{[s,t]}\equiv0$;
\item[b')] $a_{33}^{[s,t]}=\frac{g(t)}{g(s)}$, where $g$ is an arbitrary function with $g(s)\neq0$;
\item[c')] $a_{33}^{[s,t]}=\left\{\begin{array}{l}
1,\,\,\,\,\, \text{if}\,\,\,\,\, s\leq t<a,\\
0,\,\,\,\,\, \text{if}\,\,\,\,\, t\geq a,\end{array}\right.$ where $a>0$.
\end{itemize}

\textbf{Case 1.1.a.a'.} $a_{33}^{[s,t]}\equiv0.$ Consequently $a_{31}^{[s,t]}=a_{32}^{[s,t]}=0.$
Thus in this case we doesn't give new algebra.

\textbf{Case 1.1.a.b'.} $a_{33}^{[s,t]}=\frac{g(t)}{g(s)}.$ By put this value to first and
second equations of (\ref{equation5}) we give that $a_{31}^{[s,t]}=\frac{\varphi(t)}{g(s)}$ and
$a_{31}^{[s,t]}=\frac{f(t)}{g(s)}$. Thus we give the following new CEA:
$$\mathcal M_{5}^{[s,t]}=\left(\begin{array}{cccccc}
0 & 0 & 0\\[2mm]
0 & 0 & 0\\[2mm]
\frac{\varphi(t)}{g(s)} & \frac{f(t)}{g(s)} & \frac{g(t)}{g(s)}\end{array}\right).$$

\textbf{Case 1.1.a.c'.} $a_{33}^{[s,t]}=\left\{\begin{array}{l}
1,\,\,\,\,\, \text{if}\,\,\,\,\, s\leq t<a,\\
0,\,\,\,\,\, \text{if}\,\,\,\,\, t\geq a,\end{array}\right.$ where $a>0.$ Then by putting this solution
to first and second equations of (\ref{equation5}) we give that $a_{31}^{[s,t]}=\left\{\begin{array}{l}
\varphi(t),\,\,\,\,\, \text{if}\,\,\,\,\, s\leq t<a,\\
0,\,\,\,\,\, \text{if}\,\,\,\,\, t\geq a\end{array}\right.$ and $a_{32}^{[s,t]}=\left\{\begin{array}{l}
\psi(t),\,\,\,\,\, \text{if}\,\,\,\,\, s\leq t<a,\\
0,\,\,\,\,\, \text{if}\,\,\,\,\, t\geq a\end{array}\right.$ Thus we give $\mathcal M_{0}^{[s,t]}$ and the following new CEA:
$$\mathcal M_{6}^{[s,t]}=\left\{\begin{array}{l}
\left(\begin{array}{cccccc}
0 & 0 & 0\\[2mm]
0 & 0 & 0\\[2mm]
\varphi(t) & \psi(t) & 1\end{array}\right),\,\,\,if\,\,\,s\leq t<a,\\[8mm]
\left(\begin{array}{cccccc}
0 & 0 & 0\\[2mm]
0 & 0 & 0\\[2mm]
0 & 0 & 0\end{array}\right),\,\,\,if\,\,\,t\geq a.\end{array}\right.$$

Cases 1.1.b, 1.1.c and 1.2, 1.3 are more difficult. So we don't include to this paper these cases.

\section{Baric property transition}
In \cite{CLR} a notion of property transition for CEAs defined.
\begin{definition}
Assume a CEA, $E^{[s,t]}$, has a property, say $P$, at pair of
times $(s_{0},t_{0});$ one says that the CEA has $P$ property transition if there is a
pair $(s,t)\neq(s_{0}, t_{0})$ at which the CEA has no the property $P$.
\end{definition}

Denote
$$\begin{array}{ccc}
\mathcal T=\{(s, t): 0\leq s\leq t\};\\[3mm]
\mathcal T_{P}=\{(s, t)\in\mathcal T: E^{[s,t]}\,\,\, \text{has property}\,\,\, P\};\\[3mm]
\mathcal T_{P}^0=\mathcal T\backslash \mathcal T_{P}=\{(s, t)\in\mathcal T: E^{[s,t]}\,\,\, \text{has no property}\,\,\, P\}.\end{array}$$

The sets have the following meaning:\\
$\mathcal T_{P}-$the duration of the property $P;$\\
$\mathcal T_{P}^0-$the lost duration of the property $P;$\\
The partition $\{\mathcal T_{P},\mathcal T_{P}^0\}$ of the set $\mathcal T$ is called $P$ property diagram.\\
For example, if $P$=commutativity then since any evolution algebra is
commutative, we conclude that any CEA has not commutativity property
transition.

Since an evolution algebra is not a baric algebra, in general, using
Theorem \ref{barikliksharti} we can give baric property diagram. Let us do this for the above
given chains $E_{i}^{[s,t]},$ $i=0,\ldots,6.$

\begin{theorem}\label{baricprop3-dim}
\begin{itemize}
\item[(i)] (There is no non-baric property transition) The algebras $E_{i}^{[s,t]},$ $i=0,3$
are not baric for any time $(s,t)\in \mathcal T;$
\item[(ii)] (There is baric property transition) The CEAs $E_{i}^{[s,t]},$ $i=1,2,4,5$ have
baric property transition with baric property duration sets as the following\\
$\mathcal T_{b}^{(1)}=\{(s,t)\in\mathcal T: g(s)=f(s)=\pm\frac{1}{h(s)}\}\cup\{(s,t)\in\mathcal T: g(s)=-f(s)=\frac{1}{h(s)}\};$\\
$\mathcal T_{b}^{(2)}=\{(s,t)\in\mathcal T: \varphi(s)=\psi(s)=\pm1\}\cup\{(s,t)\in\mathcal T: \varphi(s)=1,\,\,\,\psi(s)=-1\};$\\
$\mathcal T_{b}^{(4)}=\{(s,t)\in\mathcal T: g(t)\neq0\};$\,\,\,\,\,\,\,\,\,\,\,\,\,\,
 $\mathcal T_{b}^{(5)}=\{(s,t)\in\mathcal T: s\leq t<a\}.$
\end{itemize}
\end{theorem}
\begin{proof}
By Theorem \ref{barikliksharti} a three-dimensional evolution algebra $E^{[s,t]}$ is baric if and only if
$a_{11}^{[s,t]}\neq0, a_{21}^{[s,t]}=a_{31}^{[s,t]}=0$ or $a_{22}^{[s,t]}\neq0, a_{12}^{[s,t]}=a_{32}^{[s,t]}=0$ or
$a_{33}^{[s,t]}\neq0, a_{13}^{[s,t]}=a_{23}^{[s,t]}=0$. By detailed checking of these
conditions over algebras $E_{i}^{[s,t]},$ $i=0,\ldots,5$ we give proof of the Theorem \ref{baricprop3-dim}.
\end{proof}

\section{Absolute nilpotent and idempotent elements transition}

In this section we answer to problem of existence of "uniqueness of absolute nilpotent element" property transition.
\begin{definition}
The element $x$ of an algebra $A$ is called an absolute nilpotent if $x^2=0.$
\end{definition}

Let $E$ be an n-dimensional evolution algebra over the field $\mathbb R$ with matrix of structural constants
$\mathcal M=(a_{ij})$, then for arbitrary $x=\sum_{i}x_{i}e_{i}$ we have
$$x^2=\sum_{j}\left(\sum_{i}a_{ij}x_{i}^2\right)e_{j}.$$
Then the equation $x^2=0$ is given by the following system

\begin{equation}\label{eqnilpotent}
\sum_{i}a_{ij}x_{i}^2=0,\,\,\,\,j=1,\ldots,n.
\end{equation}

In this section we shall solve the system (\ref{eqnilpotent}) for $E_{i}^{[s,t]}$, i=0,\ldots,6.

For a CEA $E_{i}^{[s,t]}$ with matrix $\mathcal M_{i}^{[s,t]}$ denote \\
$\mathcal T_{nil}^{(i)}=\{(s,t)\in\mathcal T: E_{i}^{[s,t]}\,\, \text{has unique absolute nilpotent}\},$
$\mathcal T_{nil}^0=\mathcal T\backslash\mathcal T_{nil}.$

Then we have the following theorem which contain answer to problem of existence
of "uniqueness of absolute nilpotent element" property transition.

\begin{theorem}
 \begin{itemize}
 \item[(1)] There CEAs $E_{i}^{[s,t]},$ $i=0,4,5$ have infinitely many of absolute
nilpotent elements for any time $(s,t)\in\mathcal T;$
 \item[(2)] The CEAs $E_{i}^{[s,t]},$ $i=1,2$ have "uniqueness of absolute nilpotent element"
 property transition with the property duration sets as the following\\
$\mathcal T_{nil}^{(1)}=\{(s,t)\in\mathcal T: h(t)\neq0, f(s)<g(s)<\frac{1}{h(s)}, -f(s)<\frac{1}{h(s)}\};$\\
$\mathcal T_{nil}^{(2)}=\{(s,t)\in\mathcal T: -1<\psi(s)<\varphi(s)<1, t<a\};$\\
$\mathcal T_{nil}^{(3)}=\{(s,t)\in\mathcal T: g_{1}(t), g_{2}(t),g_{3}(t)-\text{the signs are difference}\}.$
\end{itemize}
\end{theorem}

\begin{proof}
The proof consists the simple analysis of the solutions of the system
(\ref{eqnilpotent}) for each $E_{i}^{[s,t]},$ $i=0,\ldots,5$
\end{proof}

An element $x$ of an algebra $A$ is called idempotent if $x^2=x.$ Such points
of an evolution algebra are especially important, because they are the fixed
points (i.e. $V(x)=x$) of the evolution operator $V$. We denote by
$\mathcal Id(E)$ the set of idempotent elements of an algebra $E.$ Since
$x=\sum_{i}x_{i}e_{i}$ then the equation $x^2=x$ can be written as

\begin{equation}\label{eqidempotent}
x_{j}=\sum_{i=1}^n a_{ij}x_{i}^{2},\,\,\,\,\,\,\,\,j=1,\ldots,n.
\end{equation}

The general analysis of the solutions of the system (\ref{eqidempotent}) is very difficult. We
shall solve this problem for the CEAs $E_{i}^{[s,t]},\,\,\, i=0,\ldots,5.$
The following theorem gives the time-dynamics of the idempotent elements
for algebras $E_{i}^{[s,t]},\,\,\, i=0,\ldots,5.$

\begin{theorem}
\item[(1)] The algebra $E_{0}^{[s,t]}$ have unique idempotent $(0,0)$.
\item[(2)]
$\mathcal Id(E_{1}^{[s,t]})=\left\{\begin{array}{ll}
(0,0,0),\,\,\,\text{if}\,\,\, (s,t)\in\{(s,t)\in\mathcal T: h(t)=0\}\\[2mm]
(0,0,0), \left(\frac{h(s)}{h(t)},\frac{h(s)}{h(t)},\frac{h(s)}{h(t)}\right), \,\,\,\text{if}\,\,\, (s,t)\in\{(s,t)\in\mathcal T: h(t)\neq0\}\end{array}\right.$

\item[(3)]
$\mathcal Id(E_{2}^{[s,t]})=\left\{\begin{array}{ll}
(0,0,0),\,\,\,\text{if}\,\,\, (s,t)\in\{(s,t)\in\mathcal T: t\geq a\}\\[2mm]
\left\{(0,0,0), (1,1,1)\right\}, \,\,\,\text{if}\,\,\, (s,t)\in\{(s,t)\in\mathcal T: s\leq t<a\}\end{array}\right.$

\item[(4)]
$\mathcal Id(E_{3}^{[s,t]})=\left\{\begin{array}{ll}
(0,0,0),\,\,\,\text{if}\,\,\, (s,t)\in\{(s,t)\in\mathcal T: F(s,t)=0\}\\[2mm]
\left\{(0,0,0), (\frac{g_{1}(t)}{F(s,t)},\frac{g_{2}(t)}{F(s,t)},\frac{g_{3}(t)}{F(s,t)})\right\}, \,\,\,\text{if}\,\,\, (s,t)\in\{(s,t)\in\mathcal T: F(s,t)\neq0\}\end{array}\right.$\\
 where $F(s,t)=f_{1}(s)(g_{1}(t)+\psi(s)(g_{2}(t))^2+\varphi(s)(g_{3}(t))^2).$

\item[(5)]
$\mathcal Id(E_{4}^{[s,t]})=\left\{\begin{array}{ll}
(0,0,0),\,\,\,\text{if}\,\,\, (s,t)\in\{(s,t)\in\mathcal T: g(t)=0\}\\[2mm]
\left\{(0,0,0), (\frac{g(s)}{g(t)},\frac{g(s)f(t)}{(g(t))^2},\frac{g(s)\varphi(t)}{(g(t))^2})\right\}, \,\,\,\text{if}\,\,\, (s,t)\in\{(s,t)\in\mathcal T: g(t)\neq0\}\end{array}\right.$

\item[(6)]
$\mathcal Id(E_{5}^{[s,t]})=\left\{\begin{array}{ll}
\left\{(0,0,0), (1,\psi(t),\varphi(t))\right\}, \,\,\,\text{if}\,\,\, (s,t)\in\{(s,t)\in\mathcal T: s\leq t<a\},\\[2mm]
(0,0,0),\,\,\,\text{if}\,\,\, (s,t)\in\{(s,t)\in\mathcal T: t\geq a\}.
\end{array}\right.$
\end{theorem}

\begin{proof}
The proof contains detailed analysis of solutions of the system (\ref{eqidempotent})
for each $E_{i}^{[s,t]}$. We shall give here proof of the assertion (4) which is more
substantial.

In this case the system (\ref{eqidempotent}) has the form
\begin{equation}\label{eqidempotent3}
\left\{\begin{array}{lll}
x_{1}=f_{1}(s)g_{1}(t)(x_{1}^2+\psi(s)x_{2}^2+\varphi(s)x_{3}^2),\\[2mm]
x_{2}=f_{1}(s)g_{2}(t)(x_{1}^2+\psi(s)x_{2}^2+\varphi(s)x_{3}^2),\\[2mm]
x_{3}=f_{1}(s)g_{3}(t)(x_{1}^2+\psi(s)x_{2}^2+\varphi(s)x_{3}^2).\end{array}\right.
\end{equation}
It is clear that the system (\ref{eqidempotent})
for $E_{3}^{[s,t]}$ has a solution $(0,0,0)$.

Now we we assume that $g_{1}(t)\neq0,$ then
$f_{1}(s)(x_{1}^2+\psi(s)x_{2}^2+\varphi(s)x_{3}^2)=\frac{x_{1}}{g_{1}(t)}$
and therefore $x_{2}=x_{1}\frac{g_{2}(t)}{g_{1}(t)}, x_{3}=x_{1}\frac{g_{3}(t)}{g_{1}(t)}.$
If we put these values to first equation of (\ref{eqidempotent3}) we get
$x_{1}=\frac{g_{1}(t)}{F(s,t)}$, where $F(s,t)=f_{1}(s)(g_{1}(t)+\psi(s)(g_{2}(t))^2+\varphi(s)(g_{3}(t))^2).$
Consequently, $x_{2}=\frac{g_{2}(t)}{F(s,t)}$ and $x_{3}=\frac{g_{3}(t)}{F(s,t)}.$
Therefore the system (\ref{eqidempotent})
for $E_{3}^{[s,t]}$ has a solution $(\frac{g_{1}(t)}{F(s,t)},\frac{g_{2}(t)}{F(s,t)},\frac{g_{3}(t)}{F(s,t)})$

Proofs of the assertions $(i), i=1,2,4,5,6$ are similar analysis of the
system (\ref{eqidempotent}) for each $E_{i}^{[s,t]}$.
\end{proof}

\section{Conclusion}

In conclusion we note that in the present paper we have constructed chains of three-dimensional evolution
algebras and have studied behavior of the property to be baric for
each chains constructed in the section 3. We showed that some of the chains are
never baric. For other chains which have (baric property
transition) we defined a baric property controller function. In the last section for chains of evolution
algebras we have studied the behavior and dynamics of the set of absolute nilpotent
and idempotent elements depending on the time respectively.


\begin{resetCounters}\end{resetCounters}

\end{document}